\documentclass[12pt]{amsart}
\usepackage[utf8]{inputenc}

\usepackage[utf8]{inputenc}
\usepackage{verbatim}
\usepackage{url}

\usepackage{amsmath}
\usepackage{amssymb}
\usepackage{tikz-cd}
\usepackage{mathrsfs}
\usepackage{graphicx}
\usepackage{hyperref}
\usepackage{amsthm}
\usepackage{xcolor}
\usepackage{contour}
\usepackage{color}
\usepackage{comment}
\usepackage{stmaryrd}

\newcommand\codim{\operatorname{codim}}

\renewcommand{\AA}{{\mathbb{A}}}

\newcommand\spec{\operatorname{Spec}}
\newcommand\E{{E}}
\newcommand\G{\mathbb{G}}

\newcommand\ceq{c^{\G_m}}
\newcommand\p{\mathbb{P}}
\newcommand\Spec{\text{Spec}}
\newcommand\Proj{\text{Proj}}
\newcommand{\ra}{\rightarrow}
\newcommand{\Gan}{G_{\mathbf{a}, \mathbf{n}}}
\newcommand{\Gln}{\mathrm{GL}_\mathbf{n}}
\newcommand{\Uan}{U_{\mathbf{a}, \mathbf{n}}}
\newcommand\A{\mathbb{A}}
\newcommand\Z{\mathbb{Z}}

\newcommand{\srY}{\mathscr{Y}}
\newcommand{\srX}{\mathscr{X}}

\definecolor{VA}{RGB}{255,215,0}
\definecolor{SO}{RGB}{0,40,255}

\newtheorem*{Untheorem}{Theorem}

\newtheorem{theorem}{Theorem}[section]
\newtheorem{proposition}[theorem]{Proposition}
\newtheorem{lemma}[theorem]{Lemma}
\newtheorem{corollary}[theorem]{Corollary}

\theoremstyle{definition}
\newtheorem{definition}[theorem]{Definition}
\newtheorem{remark}[theorem]{Remark}
\newtheorem{example}[theorem]{Example}

\makeatletter
\let\@wraptoccontribs\wraptoccontribs
\makeatother

\begin{document}

\title{The integral Chow ring of weighted blow-ups}
\author{Veronica Arena}
\address{Brown University\\
         Department of Mathematics\\
         Providence \\
         RI 02906}
\email{{Veronica\_Arena@brown.edu}}

\author{Stephen Obinna}
\address{Brown University\\
         Department of Mathematics\\
         Providence \\
         RI 02906}
\email{Stephen\_Obinna@brown.edu}
\contrib[with an appendix by]{Dan Abramovich, Veronica Arena, and Stephen Obinna}
\date{\today}
\thanks{This research is supported in part by funds from BSF grant 2018193 and NSF grant DMS-2100548}
\maketitle
\begin{abstract}
We give a formula for the Chow rings of weighted blow-ups. Along the way, we also compute the Chow rings of weighted projective stack bundles, a formula for the Gysin homomorphism of a weighted blow-up, and a generalization of the splitting principle. In addition, in the appendix we compute the Chern class of a weighted blow-up.
\end{abstract}

\section{Introduction}

Let $f:\tilde Y \ra Y$ be the weighted blow-up of $X \subset Y$ with positive weights $a_1, \dots a_d$ and let $\tilde X$ be the exceptional divisor. 
For most of the paper we will assume $X,\ Y$ algebraic spaces over a field of characteristic $0$. In Section $7$ we will generalize to the case of $\srX, \ \srY$ quotient stacks by a linear algebraic group. 

Then we have the commutative diagram 

$$\begin{tikzcd}
\tilde X \arrow[r, "j"] \arrow[d, "g"] & \tilde Y \arrow[d, "f"] \\
X \arrow[r, "i"] & Y
\end{tikzcd}$$
which is not Cartesian, unlike the ordinary blow-up case (an example of this can be found in \cite[Remark 3.2.10]{quek-rydh-weighted-blow-up}). 

In the case of a classical blow-up, a description of the ring $A^*(\tilde Y)$ and of its $A^*(Y)$-module structure is given in \cite[Exercise 8.3.9]{fulton} or \cite[Proposition 13.12]{eisenbud_harris}. 
The purpose of this paper is to give a similar description for the Chow ring of a weighted blow-up. 

The formula will follow from the exact sequence in the theorem below, generalizing the key sequence in \cite[Proposition 6.7(e)]{fulton}.

\begin{Untheorem}[\ref{sequence} Key sequence]
Let $X$, $Y$, $\tilde{X}$, $\tilde{Y}, f$ be as above, then we have the following exact sequence of Chow groups. 
$$A^*(X) \xrightarrow{(f^!, - i_*)} A^*(\tilde X) \oplus A^*(Y) \xrightarrow{j_*+f^*} A^*(\tilde Y) \ra 0.$$
Further, if we use rational coefficients, then this becomes a split short exact sequence with $g_*$ left inverse to $(f^!, -i_*)$.
$$0\to A^*(X,\mathbb{Q}) \xrightarrow{(f^!, -i_*)} A^*(\tilde X,\mathbb{Q}) \oplus A^*(Y,\mathbb{Q}) \xrightarrow{j_*+f^*} A^*(\tilde Y,\mathbb{Q}) \ra 0.$$
\end{Untheorem} 

Note that since our blow-up diagram is not Cartesian, the codomain of $f^!$ is $A^*(X\times_Y \tilde{Y})$, but $\tilde{X}$ is the reduction of $X\times_Y \tilde{Y}$ so we can identify their Chow groups. 

 Moreover, when working with integral coefficients, the sequence is no longer exact on the left as shown in Example \ref{non exactness}.
Passing to rational coefficients however, allows us to maintain exactness on the left and to define a left inverse of $(f^!,-i_*)$ via $g_*$. In fact, it is enough to pass to $\Z[\frac{1}{a_1},\dots \frac{1}{a_d}]$-coefficients.

From the sequence, we can get the following description of $A^*(\tilde Y)$.

\begin{Untheorem}[\ref{main theorem 4} Chow ring of a weighted blow-up] If $\tilde{Y}\to Y$ is a weighted blowup of $Y$ at a closed subvariety $X$, then the Chow ring $A^*(\tilde{Y})$ is isomorphic \emph{as a group} to the quotient $$A^*(\tilde Y)\cong\frac{(A^*(X)[t])\cdot t \oplus A^*(Y)}{( (({P(t)-P(0)})\alpha,-i_*(\alpha)), \forall\alpha\in A^*(X))}$$ with $P(t)=c_{top}^{\G_m}(\mathcal{N}_XY)(t)$ and $[\tilde X]=-t$.

The multiplicative structure on $A^*(\tilde Y)$ is induced by the multiplicative structures on $A^*(X)$ and $ A^*(Y)$ and by the pullback map in the following way $$(0,\beta) \cdot (t,0)=(i^*(\beta)t,0).$$

Equivalently $A^*(\tilde Y)$ can be expressed as a quotient of the fiber product 
$$ \frac{A^*(Y)\times_{A^*(X)} A^*(X)[t]}{((i_*\alpha, P(t)\alpha) \ \forall \alpha \in A^*(X))} $$ with $i^*: A^*(Y) \ra A^*(X)$ and $A^*(X)[t] \ra A^*(X)$ given by evaluating $t$ at $0$.

\end{Untheorem}

In order to use the key sequence, we need to give a presentation for the Chow ring of the exceptional divisor $\tilde X$.
 
In the classical case, $\tilde X$ is a projective bundle over $X$ and the Chow ring of a projective bundle can be described via the formula \cite[Theorem 9.6]{eisenbud_harris}.
In the case of a weighted blow-up, the exceptional divisor is a projective stack bundle, i.e.,the projectivization of a weighted affine bundle (Definitions \ref{bundledef}, \ref{pbundledef}). In Section 3 we define the top $\G_m$-equivariant Chern class for a weighted affine bundle $E$ in terms of its homogeneous pieces as $$c^{\G_m}_{top}(E)=\prod c^{\G_m}_{top}(E_i)=\prod_{i}(c_{n_i}(E_i)+a_itc_{n_i-1}(E_i)+...+a_i^{n_i}t^{n_i})$$ 
and we give a formula for the integral Chow ring of a projective stack bundle (which was proven for rational coefficients in \cite[Lemma 2.10(b)]{mustata-mustata}).

\begin{Untheorem}[\ref{WeighedProjectiveBundleFormula} Weighted projective bundle formula] Let $\E$ be a weighted, affine bundle over $X$ of rank $n$. Let $c_{top}^{\G_m}(E)(t)$ be its $\G_m$-equivariant top Chern class.  
Then $$A^*(\mathcal{P}(\E)) \cong \frac{A^*(X)[t]}{c_{top}^{\G_m}(E)(t)}.$$
\end{Untheorem}

Finally, to have a complete description of the exact sequence in Theorem \ref{sequence}, we need the appropriate generalization for the excess intersection formula \cite[Theorem 6.3]{fulton}. Unlike the case of an ordinary blow-up, in the weighted blow-up case we don't have an excess bundle and we describe $f^!$ as the multiplication by a difference quotient of the top $\G_m$ equivariant Chern class of the normal bundle.

\begin{Untheorem}[\ref{Gysin homomorphism for weighted blow-up} Weighted key formula]
Let $X$, $Y$, $\tilde{X}$, $\tilde{Y}, f$ be as above. Let us identify $A^*(\tilde X) \cong A^*(X)[t]/P(t)$ with $P(t)=c_{top}^{\G_m}(\mathcal{N}_XY)(t)$.
Then we have the following formula for the Gysin homomorphism $f^!:A^*(X) \ra A^*(\tilde X)$. $$
 f^!(\alpha)=\frac{P(t)-P(0)}{t}\alpha.$$
\end{Untheorem}

The proof of our formula for the Gysin homomorphism relies on a generalization of the splitting principle, Theorem \ref{splitting principle}

\begin{Untheorem}[\ref{splitting principle} The splitting principle]Let $T$ be the standard maximal torus in $\Gln$. Then the map $X'' \ra X$ in the following diagram $$\begin{tikzcd}
X'' \arrow[r] \arrow[d] & BT \arrow[d]\\
X' \arrow[r] \arrow[d] & B\Gln \arrow[d] \\
X \arrow[r] & B\Gan
\end{tikzcd}$$ induces an injection of Chow rings $A^*(X) \hookrightarrow A^*(X'')$ via pullback. 
\end{Untheorem}

Here $\Gan$ is the structure group of the weighted affine bundle $E$ and $\Gln$ is the structure group of its associated weighted vector bundle.
Note that the upper square of the diagram is equivalent to the classical splitting principle in \cite{fulton} or \cite{totaro}.

\subsection{Acknowledgements}

We would like to thank to Dan Abramovich for his invaluable help and guidance, as well as Jarod Alper, Paolo Aluffi, Martin Bishop,  Samir Canning, Andrea Di Lorenzo, Giovanni Inchiostro, Patrick Jefferson, Michele Pernice and Ming Hao Quek for insightful conversations.

\setcounter{tocdepth}{1}
\tableofcontents

\section{Equivariant Intersection Theory}

From now on $Y,X$ will be smooth quasi-separated algebraic spaces, of finite type over a field $k$ of characteristic 0, with a $\G_m$ action. We will also assume the $\G_m$ action is trivial on $X$.

Let us first recall the definitions in \cite{EG} of equivariant Chow groups, for a linear algebraic group $G$. 

\begin{definition} \label{definition of G-equivariant Chow groups} \cite[Definition-Proposition 1]{EG}
Let $Y$ be a $d$-dimensional quasi-separated algebraic space of finite type over a field $k$, together with a $G$ action. Let $g$ the dimension of $G$. 
The $i$-$th$ $G$-equivariant Chow group of $Y$ is defined as $$A_i^{G}(Y):=A_{i+l-g}(Y \times U/G)$$ where $U$ is an open subspace of an $l$-dimensional representation, on which $G$ acts freely and whose complement has codimension greater than $d-i$. 
\end{definition}

In this article we will mostly use this definition in the particular case of a $\G_m$ action. In particular, this leaves us with very convenient choices for representations: $V$ will be an $l$-dimensional vector space with the standard $\G_m$ action with weight $1$ and $U= V \smallsetminus \{0\}$.

\begin{example} \label{id}
$A^*_{\G_m}(X)\cong A^*(X)[t]$. 

Indeed, since the $\G_m$ action is trivial on $X$
\begin{align*}A^i_{\G_m}(X)&=A_{d-i}^{\G_m}(X)=A_{d-i+l-1}\left({X \times U}/{\G_m}\right)=A^i(X \times \p^{l-1})\\&=\bigoplus_{k=0}^iA^k(X)t^{i-k}\end{align*} with $t=c_1(\mathcal{O}_{\p^{l-1}}(1))$ and the isomorphism follows. 
\end{example}

\begin{definition}\cite[Definition 1]{EG}
Let $\E$ be a $G$ equivariant vector bundle over $Y$. The equivariant Chern classes of $E$ are the operators $c_j^{G}: A_i^{G}(Y) \to A^{G}_{i-j}(Y) $  with $c^{G}_j(\E) \cap \alpha = c_j\left( {\E \times U}/{G}\right) \cap \alpha \in A_{i+l-j-g}(Y \times U /G)=A^{G}_{i-j}(Y)$.
\end{definition}

As Molina-Rojas and Vistoli remark in \cite[Section 2]{MV}
many of the standard properties of Chow groups still hold in the equivariant case. Below we collect some that will be used later.
\begin{proposition} \label{remarks}
The following are true.

\begin{enumerate}
\item The first Chern class of the tensor product of line bundles is the sum of the first Chern classes of each line bundle.

\item Let E be $G$-equivariant a vector bundle over $Y$ and $f:Y' \to Y$ a map such that $f^*\E$ has a filtration of $G$-equivariant vector bundles $f^*\E \supset F_r \supset ... \supset F_0=0$. Let $E_i=F_{i}/F_{i-1}$.
Then $$c^G(f^*\E)=\prod_i c^G(\E_i).$$

\item If $Z$ is a closed $G$-invariant subscheme of $Y$, we have the following exact sequence $$A^*_{G}(Z) \to A^*_{G}(Y) \to A^*_{G}(Y \smallsetminus Z) \to 0.$$ 

\item Let $\pi: E\ra Y$ be a vector bundle over $Y$, $s_0: Y \to E$ the zero section. Then the Gysin pullback map $s_0^*: A^*_{G}(E) \to A^*_{G}(Y)$ is an isomorphism inverse of $\pi^*$.

\end{enumerate}
\end{proposition}
\begin{proof}
We will prove $(3)$. The proof of the remaining ones is analogous.
Let $U$ have dimension $l$ high enough that $A_i^{G}(Y)$ is defined as $A_{i+l-g}(Y \times U/ G)$. 
Then, also by definition, we have $A_i^{G}(Z):=A_{i+l-1}(Z \times U/G)$. In particular, $Z \times U/G$ and $Y \times U/G$ are algebraic spaces, and the localization sequence $$A^*
(Z \times U /G) \ra A^*(Y \times U/G ) \ra A^*(Y\times U/G \smallsetminus Z \times U /G) \ra 0$$ is exact.
Therefore statement $(3)$ holds as well. 
\end{proof}

Another proposition, that will be very useful later, is \cite[Lemma 2.2]{MV}, for which we will quote the statement and the proof.

\begin{lemma}\cite[Lemma 2.2]{MV} \label{lemma 2.2}
Let $G$ be an affine linear group acting on a smooth scheme $Y$. Let $\pi:E \ra Y$ be a $G$-equivariant vector bundle of rank $n$. Call $E_0 \subset E$ the complement of the zero section $s: Y \ra E$. Then the pullback homomorphism $\pi|_{E_0}^*:A^*_G(Y) \ra A^*_G(E_0)$ is surjective, and its kernel is generated by the top Chern class $c_n^G(E) \in A^n_G(Y)$.
\end{lemma}

\begin{proof}
Consider the diagram:
$$\begin{tikzcd}
&A^*_G(Y)\arrow[dr,"\pi|_{E_0}^*"]\\
A^*_G(Y) \arrow[ur]\arrow[r,"s_*"]& A^*_G(E) \arrow[u,"s^*"]\arrow[r]& A^*_G(E_0)\arrow[r]&0,
\end{tikzcd}$$
where the bottom is the localization sequence. Since $s^*$ is an isomorphism, inverse to $\pi^*$, we see that $\pi|_{E_0}^*$ is surjective with kernel generated by the image of $s^*s_*$. By self-intersection formula, $s^*s_*$ is multiplication by $c_n^G(E)$.
\end{proof}

\section{Chow groups of weighted projective stack bundles}

In this section we will give formula for the Chow ring of a weighted projective stack bundle. Weighted projective stack bundles appear as the exceptional divisor of a weighted blow-up. We start by computing the Chern classes of weighted affine bundles, and then show we can apply Lemma \ref{lemma 2.2} to them. A similar formula for rational coefficients appears in  \cite[Theorem 2.10 (b)]{mustata-mustata}.

\begin{definition}
An affine bundle is a smooth affine morphism $E\to X$ such that $E$ is, locally in the smooth topology, isomorphic to $X\times \mathbb{A}^n$.
\end{definition}
\begin{definition}\cite[Definition 2.1.3]{quek-rydh-weighted-blow-up}\label{bundledef}
A weighted affine bundle is a $\G_m$ equivariant affine bundle $E \ra X$ where locally in the smooth topology $\G_m$ acts linearly on $\mathbb{A}^n$ with positive weights $a_1,...,a_n \in \mathbb{Z}$. 
\end{definition}

We will show in Lemma \ref{decompositionofGan} that the structure group is special, so weighted affine bundles over a scheme will be Zariski-locally-trivial.

It will sometimes be convenient to emphasize the \emph{distinct} weights of a weighted affine bundle. When we do this we will list the distinct weights as $a_1,\dots,a_r$, and use $n_i$ to refer to the dimension of the subspace of $\mathbb{A}^n$ where the action has weight $a_i$.

\begin{definition}
    A weighted vector bundle is a weighted affine bundle whose underlying $\G_m$ space is a vector bundle. Equivalently, it is a weighted affine bundle with linear transition functions.
\end{definition}

Notice that our terminology is slightly different from that of \cite{quek-rydh-weighted-blow-up}. What they call twisted/untwisted weighted vector bundles, we call weighted affine/vector bundles respectively. Also note that the action on a weighted vector bundle makes it a sum of vector bundles with homogeneous actions.

\begin{definition}\cite[Definition 2.1.5]{quek-rydh-weighted-blow-up}\label{pbundledef}
A weighted projective stack bundle over $X$ is the stack-theoretic Proj of a graded algebra corresponding to a  weighted affine bundle with strictly positive weights. Precisely, if $R$ is a graded algebra such that $E = \Spec _X(R)$ then $\Proj_X(R)=[\Spec_X(R)\smallsetminus V(R_+)/\G_m]$.
\end{definition}

\subsection{Equivariant Chern classes of a weighted line bundle}

Let us denote with $L \to X$ a line bundle over $X$ with the trivial action. Let us denote with $L^a$ the same underlying line bundle, endowed with the weight $a$ $\G_m$ action. This is a notation we will adopt only for subsections $3.1$ and $3.2$, but abandon later as the weight of the $\G_m$ action will be clear from context.

Some of the following lemmas are likely already known, but are stated and proven for completeness as we couldn't find specific references.

\begin{lemma} \label{eclb}
Let $L^a$ be a $\G_m$ equivariant line bundle over $X$ with weight $a$, then the first equivariant Chern class of $L^a$ is $\ceq_1(L^a)=c_1(L)+at$ via the identification in Example
\ref{id}.
 \end{lemma}

 \begin{proof}
 We know that 
   $L^a=L^a \otimes \mathcal{O}_X = L \otimes {\mathcal{O}_X}^a$.
   
   In particular, $$c_1^{\G_m}(L^a)=c_1^{\G_m}(L \otimes \mathcal{O}_X)=c_1^{\G_m}(L)+ac_1^{\G_m}({\mathcal{O}_X}^1).$$
   Now, since the action on $L$ is trivial,  $(L \times U )/\G_m=L \times \p^{l-1}$ and since $A^1(X \times \p^{l-1})=A^1(X) \oplus A^0(X)t$, we get $$c_1^{\G_m}(L)=c_1(L \times \p^{l-1})=c_1(L) \in A^1(X). $$ We only need to prove $c_1({\mathcal{O}_X}^1)=t$.
   
   Let us consider the projection to a point $P$, $f:X \to P$. The map defines a graded ring homomorphism $f^*: A^*_{\G_m}(P) \ra A^*_{\G_m}(X)$, i.e.,a map $f^*: \mathbb{Z}[t] \ra A^*(X)[t]$ defined by $1 \mapsto 1$ and $t \mapsto t$.
   
Now ${\mathcal{O}_X}^1=f^*({\mathcal{O}_P}^1)$ and $$c_1^{\G_m}({\mathcal{O}_X}^1)=c_1^{\G_m}(f^*({\mathcal{O}_P}^1))=f^*c_1^{\G_m}({\mathcal{O}_P}^1).$$ Therefore it is enough to prove $c_1^{\G_m}({\mathcal{O}_P}^1)=t$. 

By definition $c_1^{\G_m}({\mathcal{O}_P}^1)=c_1(\mathcal{O}_P \times U/ \G_m)$ as a bundle over $U/ \G_m$, with $U / \G_m=\mathbb{A}^2 \smallsetminus (0,0)/ \G_m =\p^1$. 

Now, a nonzero section $s: U \ra U \times\mathcal{O}_P / \G_m$ is given by $(x_0, x_1) \mapsto (x_0, x_1, x_0)$ and intersects the zero section $(x_0,x_1,0)$ in $x_0=0$, which gives us $\mathcal{O}_{\p^1}(1)$, whose first Chern class is $t$ in $A^1(U \times\mathcal{O}_P / \G_m)=A^1(\p^1)$, as desired. 
 \end{proof}

 \subsection{Equivariant Chern classes of homogeneous bundles}

 \begin{proposition}[{Homogeneous bundles admit a splitting with line bundles}]\label{filtration}
Let $\E^a$ be a rank $n$ vector bundle over $X$ with $\G_m$ acting homogeneously with weight $a$ on it. Then there exists $f:X' \to X$ such that $f^*\E^a$ has a filtration $$f^*\E^a \supset F^a_{n} \supset ... \supset F^a_0=0$$ with $\G_m$-equivariant line bundle quotients $L^a_j=F^a_{j+1}/F^a_{j}$ and $f^*$ is injective. 
\end{proposition}

\begin{proof}
Let us consider the underlying bundle $E$. Then by splitting construction \cite[page 51]{fulton} there is a map $f:X'\to X$ with a filtration $f^*E = F_{n} \supset ... \supset F_0=0$ with line bundle quotients.

These bundles naturally have the structure of $\G_m$-equivariant vector bundles with weight $1$. By replacing the weight $1$ action with a weight-$a$ action we get the desired sequence. 
\end{proof}


\begin{corollary} \label{ceqvb}
 Let $E^a$ be a homogeneous $\G_m$-equivariant vector bundle of rank $n$ with weight $a$, $E$ the underlying vector bundle endowed with the trivial $\G_m$ action. Then the top equivariant Chern class in $A^*_{\G_m}(X)=A^*(X)[t]$ is $$c_{n}^{\mathbb{G}_m}(E^a)=c_{n}(E)+atc_{n-1}(E)+...+a^{n}t^{n}.$$
\end{corollary}

\begin{proof}
 Let $f:X' \ra X$ as in Proposition \ref{filtration}. Then $c^{\G_m}_{n_a}(f^*E^a)=\prod_{i=1}^{n_a} c^{\G_m}_1(L_i^a)$. By Lemma $\ref{eclb}$ $c^{\G_m}_1(L_i^a)=c_1(L_i)+at$.
 
     Therefore $c_{n}^{\mathbb{G}_m}(f^*E^a)=c_{n}(f^*E)+atc_{n-1}(f^*E)+...+a^{n}t^{n}$. By injectivity of $f^*$ we are done.
\end{proof}

\subsection{Chern classes of weighted affine bundles}

\begin{proposition}\label{Filtration2}
{Let $\E$ be a weighted affine bundle over $X$. Let $0<a_1<...<a_r$ the different weights of the $\G_m$ action. Then there exist subbundles $F_i$ such that $$\E \supset F_r \supset ... \supset F_1 \supset 0$$ with well defined quotients $\E_i=F_i/F_{i-1}$ which are homogeneous vector bundles with weight $a_i$. } 
\end{proposition}

\begin{proof} {
Let $E=\spec_{X}(R)$ and $\{U_i\}$ be a cover for $X$ such that $E|_{U_i}$ is the trivial bundle. Then
we have graded isomomorphisms $$\alpha_i: R|_{U_i} \to \mathcal{O}_{U_i}[x_{i,1}^{(a_1)},\dots,x_{i,n_1}^{(a_1)},x_{i,1}^{(a_2)},\dots,x_{i,n_2}^{(a_2)},\dots,x_{i,1}^{(a_r)},\dots,x_{i,n_r}^{(a_{r})}]$$ with $x_{i,1}^{(a_h)},...,x_{i,n_h}^{(a_h)}$ having weight $a_h$. A general transition map $$\alpha_{ij}=\alpha_j|_{U_{ij}} \circ \alpha_i|_{U_{ij}}^{-1}: \mathcal{O}_{U_ij}[x_{i,1}^{(a_1)},...,x_{i,n_r}^{(a_{n_r})}] \to \mathcal{O}_{U_{ij}}[x_{j,1}^{(a_1)},...,x_{j,r}^{(a_{n_r})}] $$ will map $x_{i,l}^{(a_h)}$ to a homogeneous polynomial of degree $a_h$.

Now let $F_k|_{U_i}:=\alpha_i^{-1}(\mathcal{O}_{U_i}[x_{i,1}^{(a_1)},...,x^{(a_k)}_{i,n_k}])$ be the locus where $\G_m$ acts with weights smaller or equal to $a_k$. This defines a subbundle of $E$. Moreover the quotients $F_k/F_{k-1}$ are well defined. Indeed, while these are affine bundles, they are locally isomorphic to vector bundles so we can at least take quotients locally. By construction, taking quotients locally gives us bundles consisting only of the weight $a_k$ pieces of $E$, and since the lower degree pieces have been reduced to 0, we are left with linear transition functions making the $F_k/F_{k-1}$ homogeneous bundles of weight $a_k$, as needed. 
}\end{proof}

\begin{definition}\label{affine Chern class}
For an affine bundle $E$, with $E_i$ as in Proposition \ref{Filtration2}  we define the $G$-equivariant total Chern class of $E$ as $c^{G}(E)=\prod c^{G}(E_i).$
\end{definition}

\begin{proposition}\label{weighted self intersect} Let $E$ be a weighted affine bundle over $X$ as above. Let $N_XE$ be the (non-weighted) normal bundle of $X$ in $E$. Then, with $X\hookrightarrow E$  the zero section, we have $$N_XE\cong E_1 \oplus ... \oplus E_r.$$ 
\end{proposition}

\begin{proof}
    Let $I$ be the ideal sheaf of $X$ in $E$, then $N_XE \cong \spec_E(\oplus I^n/I^{n+1})$ is a vector bundle over $X$ with the same rank as $E$ and we need only determine the transition maps. 
    
    Letting $\alpha_{ij}$ be the transition maps of $E$, the corresponding transition maps for $N_XE$ are induced by the $\alpha_{ij}$ in the natural way. Because of the quotients, all non-linear terms of $\alpha_{ij}$ become zero and the surviving linear terms are exactly the transition functions of $E_1 \oplus ... \oplus E_r$. 
\end{proof}

\begin{corollary}
Lemma \ref{lemma 2.2} also holds in the case where $E$ is a weighted affine bundle.
\end{corollary}
\begin{proof}
Note that by Definition \ref{affine Chern class} and Proposition $\ref{weighted self intersect}$, we have $c_i^G(E)=c_i^G(N_XE)$. 
The rest of the proof is the same as in Lemma \ref{lemma 2.2}, noting that $\pi^*:A^*(X)\to A^*(E)$ is still an isomorphism by \cite[\href{https://stacks.math.columbia.edu/tag/0GUB}{Tag 0GUB}]{stacks-project}, and $s^*:A^*(E)\to A^*(X)$ is still it's inverse.
\end{proof}

 \subsection{The Chow ring of a weighted projective stack bundle}


\begin{theorem}[weighted projective bundle formula] \label{WeighedProjectiveBundleFormula}Let $\E$ be a weighted, affine bundle over $X$ of rank $n$. Let $c_{top}^{\G_m}(E)(t)$ be its $\G_m$-equivariant top Chern class.  
Then $$A^*(\mathcal{P}(\E)) \cong \frac{A^*(X)[t]}{c_{top}^{\G_m}(E)(t)}.$$
\end{theorem}
\begin{proof}
Note that, by definition, $A^*(\mathcal{P}(E))=A^*([E_0/\G_m])=A^*_{\G_m}(E_0)$.
By Lemma \ref{lemma 2.2} we only need to prove that the image of $\ceq_n(\E)$ via the identification in Example \ref{id} is  $\prod_{i}(c_{n_i}(E_i)+a_itc_{n_i-1}(E_i)+...+a_i^{n_i}t^{n_i})$.

By Corollary \ref{ceqvb}, it follows $c^{\G_m}_{n}(E)=\prod c^{\G_m}_{n_i}(E_i)=\prod_{i}(c_{n_i}(E_i)+a_itc_{n_i-1}(E_i)+...+a_i^{n_i}t^{n_i})$ and we are done. 
\end{proof}
 
 Below there are some (familiar) special cases.
 
 \begin{example}[The Chow ring of $\mathcal{P}(a_1,...,a_n)$]\cite[Lemma 4.7]{inchiostro}
 We can consider $\mathcal{P}(a_1,...,a_n)$ as a weighted projective bundle over a point. In particular, we have that $\mathcal{P}(a_1,...,a_n)$ splits into $n$ trivial bundles with weights $a_1,...,a_n$. Each of these line bundles will have the first Chern class equal to zero. It follows that $$A^*(\mathcal{P}(a_1,...,a_n))=\frac{\mathbb{Z}[t]}{a_1...a_nt^{n+1}}.$$
 \end{example}
 
 \begin{example}[The Chow ring of a classical projective bundle] \cite[Theorem 9.6]{eisenbud_harris}
 Let $E$ be a vector bundle of rank $n$ over $X$ in the classical sense.
In this case, we have $\G_m$ acting homogeneously on the whole space with weight $1$. In particular $$A^*(\mathcal{P}(E))=\frac{A^*(X)[t]}{c_n(E)+c_{n-1}(E)t+...+t^n}.$$
 \end{example}

 \begin{example} As a toric example, this can be recovered as a consequence of \cite[Theorem 2.2]{Iwanari}.
 The exceptional divisor $\tilde X$ of the weighted blow-up of $X=V(x_1,...,x_{n})$ in $\p^{m+n}$, with (possibly equal) weights $a_1,...,a_n$.
 
 In this case, the Chow ring of the base will be $A^*(X)=A^*(\p^{m})=\frac{\mathbb{Z}[x]}{(x^{m+1})}$.
 
 The normal cone of $X$ in $\p^{n+m}$, $N_X\p^{n+m}$ will split into the sum of $n$ copies of $\mathcal{O}_{\p^{m}}(1)$, on each one of which $\G_m$ will act with weight $a_i$. We will denote the normal cone together with the $\G_m$ action with $N_{a_1,...,a_n}\p^{n+m}$
 
 Now $c_1(\mathcal{O}_{\p^{m}}(1))=x \in \frac{\mathbb{Z}[x]}{x^{m+1}}$. Therefore $c_{n}^{\G_m}(N_X\p^{n+m})=\prod(x+a_it)$ and $$A^*(\tilde X)=A^*(\mathcal{P}(N_{a_1,...,a_n}\p^{n+m}))=\frac{\mathbb{Z}[x,t]}{(x^{m+1}, \prod_{i=1}^n(x+a_it))}.$$
 \end{example}

\section{The splitting principle}

The goal of this section is to prove Theorem \ref{splitting principle}, an analog of the splitting principle. We start by proving some facts about structure groups and classifying spaces. Using those results we construct, for any weighted affine bundle $E\to X$, a map $X'\to X$ that allows us to pull back our affine bundle to a vector bundle $E' \to X'$ with $A^*(X')\cong A^*(X)$.



\subsection{Structure groups}

From here on, let $E\to X$ be an affine bundle with fibers isomorphic to an affine space $V$, $\mathbf{n}=(n_1,...,n_r)$ the dimensions of its homogeneous parts and $\mathbf{a}=(a_1,...,a_r)$ their \emph{distinct} weights. Let $\Gan$ be the group $Aut_{\G_m}(V)$ of $\G_m$-equivariant automorphisms of $V$ and $\Gln=\prod GL(n_i)$. Moreover, define $V\Gan= [V/\Gan]$ and $V\Gln=[V/\Gln]$.

\begin{lemma} \label{decompositionofGan}
 There is a surjective group homomorphism $\Gan \to \Gln$ and a section $\Gln \to \Gan$. The kernel of the surjection is a unipotent group $U_{\mathbf{a},\mathbf{n}}$. This implies that $\Gan$ is special.
\end{lemma}

\begin{proof}
In \cite[Section 2.1.7]{quek-rydh-weighted-blow-up}, the authors offer an explicit description of $\Gan$. In fact, they decompose $\Gan$ in the recursive semidirect product
$$\Gan =(\textrm{GL}_{n_r} \times G_{\mathbf{a'}, \mathbf{n'}})\ltimes \G_a^{n_rN_r}$$
with $\mathbf{a'}=(a_1,...,a_{r-1})$, $\mathbf{n'}=(n_1,...,n_{r-1})$, and $N_r$ the dimension of $a_r^\textrm{th}$ degree piece of a graded polynomial algebra with free variables $\{x_{i,j}: 1 \le i \le r-1, \  1\le j\le n_i\}$ where $x_{i,j}$ is given weight-$a_i$.

Unraveling the recursion gives:
$$\Gan=\left(\textrm{GL}_{n_r}\times\dots\times \left(\left(\textrm{GL}_{n_1} \times \{1\}\right)\ltimes \G_a^{n_1N_1}\right)\ltimes\dots\right)\ltimes \G_a^{n_rN_r}=$$ $$=\left(...\left(\Gln\ltimes \G_a^{n_1N_1}\right)\ltimes\dots\right)\ltimes \G_a^{n_rN_r}.$$
This expression provide us with the desired surjection and section. The kernel is a successive extension of additive groups and so is unipotent.
\end{proof}

\subsection{Lemmas on classifying spaces}

\begin{lemma}\label{fibersquareofclassifyingspaces}
Given a semi-direct product of groups $G=L \ltimes U$,
we get a Cartesian diagram: $$\begin{tikzcd}
\{*\} \arrow[r] \arrow[d] & BL \arrow[d] \\
BU \arrow[r] & BG.
\end{tikzcd}$$
\end{lemma}

\begin{proof}
The datum of a map $T\to BL\times_{BG}BU$ consists of a $U-$bundle $P_U\to T$, an $L-$bundle $P_L\to T$, and an isomorphism of $G-$bundles $\phi: (P_U\times G)/U \to (P_L\times G)/L$.
We will prove that the bundles $P_U$ and $P_L$ are trivial and $\phi$ is unique up to unique isomorphism, so that $BL\times_{BG}BU\cong \{*\}$.

Indeed, transition functions on $P_U$ (respectively on $P_L$) are given by multiplications by elements of $U$ (respectively of $L$). In particular, the isomorphism of $G$- bundles $(P_U\times G)/U \to (P_L\times G)/L$ forces the transition functions to be in the intersection $L \cap U = \{e\}$. 
This implies $P_U$ and $P_L$ are trivial bundles and from now on we will denote $(P_L\times G)/L \cong (P_U\times G)/U$ by $P$.
We identify automorphisms of the trivial $G$-bundle with elements of $G$ through the natural isomorphism, and do the same for $U$-bundles and $L$-bundles. Given 2 automorphisms, $g,g'$ of the trivial $G$-bundle there are unique $l\in L$ and $u\in U$ such that $lg=g'u$, making the following diagram commute 
$$\begin{tikzcd}
P \arrow[r, "l"] \arrow[d, "g'"] & P \arrow[d,"g"] \\
P \arrow[r,"u"] & P.
\end{tikzcd}\ $$
We have shown that for any choice of two objects in $BU \times_{BG}BL$ there is a unique isomorphism between them, as desired. 
\end{proof}

\begin{corollary} \label{Uanbundle}
$B\Gln\to B\Gan$ is a $\Uan$ bundle, specifically there is a Cartesian diagram $$\begin{tikzcd}
\Uan \arrow[r] \arrow[d] & B\Gln \arrow[d] \\
\{*\} \arrow[r] & B\Gan.
\end{tikzcd}$$
Consequently, given a morphism $X\to B\Gan$ the fiber product $X\times_{B\Gan} B\Gln \to X$ is a $\Uan$ bundle.
\end{corollary}

\begin{proof}
Applying Lemma \ref{fibersquareofclassifyingspaces} to $\Gan,\Gln,\Uan$ as in Lemma \ref{decompositionofGan} and appending on the left the Cartesian diagram coming from the standard presentation of $B\Uan$:
$$\begin{tikzcd}
\Uan \arrow[r] \arrow[d] & \{*\}\arrow[d]\\
\{*\} \arrow[r] & B\Uan,
\end{tikzcd}$$
we get the Cartesian diagram:

$$\begin{tikzcd}
\Uan \arrow[r] \arrow[d] & \{*\} \arrow[r] \arrow[d] & B\Gln \arrow[d] \\
\{*\} \arrow[r] & B\Uan \arrow[r] & B\Gan
\end{tikzcd} $$ as desired.

Then $X\times_{B\Gan} B\Gln \to X$ is the pullback of a $\Uan$ bundle and hence is a $\Uan$ bundle.
\end{proof}

\begin{lemma} \label{GroupLemma2ElectricBoogaloo}
Let $L$ be a subgroup of a group scheme $G$ acting on a scheme $V$. Then the following diagram is Cartesian:
$$ \begin{tikzcd}
\left[V/L\right] \arrow[r] \arrow[d] & BL \arrow[d]\\
\left[V/G\right] \arrow[r] & BG.
\end{tikzcd}$$ 
\end{lemma}

\begin{proof}
An object over a scheme $S$ in $[V/G] \times_{BG}BL$ is a triple $(P, Q, \alpha)$ where $P\ra S$ is a $G$-torsor, together with a $G$-equivariant map to $V$ $$\begin{tikzcd}
P \arrow[d] \arrow[r, "\phi"] & V; \\
S
\end{tikzcd}$$ $Q \ra S$ is a $L$-torsor; $\alpha$ is an isomorphism of $G$-torsors $ P \xrightarrow{\alpha} Q \times G / L$.

Given such an object we can construct an object in $[V/L]$ by considering the $L$-torsor $Q \ra S$ together with the map $\psi: Q \ra V$ defined as follows $$\begin{tikzcd} Q \arrow[d] \arrow[r] & Q \times G / L \arrow[r, "\alpha^{-1}"]& P \arrow[r,"\phi"] & V. \\ S\end{tikzcd}$$
To verify this is indeed an object of $[V/L]$, we need to prove that $\psi$ is $L$-equivariant. 
Now, $\phi$ and $\alpha^{-1}$ are $G$-equivariant, and in particular $L$-equivariant. Moreover the quotient map $Q \ra Q \times G/L$ maps an element $ql$ to $ [ql,e]=[qll^{-1}, le]=[q,l]$. But $L$ acts on $Q \times G/L$ through its inclusion into $G$, hence $ql \mapsto [q,e]l$ as desired. 

On the other hand given a $L$-torsor $Q \ra S$ together with a $L$-equivariant map $\psi: Q \ra V$ in $[V/L]$, we can construct the triple $(Q \times G/L, Q, id)$ as an object of $[V/G] \times_{BG}BL$. In order for $Q \times G/ L $ to be an object in $[V/G]$, we must equip it with a $G$-equivariant map $Q \times G/L \ra V$ or, equivalently, with a $G$-equivariant, $L$-invariant map $Q \times G \ra V$. The map $\Phi: Q \times G \ra V$ defined by $(q,g) \mapsto \psi(q)g$ is $L$-invariant with respect to the action $l \cdot (q,g)=(ql,l^{-1}g)$ we are quotienting by, indeed $(ql,l^{-1}g) \mapsto \psi(ql)l^{-1}g=\psi(q)ll^{-1}g=\psi(q)g$. Moreover $\Phi$ is $G$-equivariant: $\Phi((q,g) \cdot h)=\psi(q)gh=(\psi(q)g) \cdot h$. 
The verification that the functors defined are indeed inverses is standard and will be omitted.

\end{proof}

\subsection{The Splitting Principle}

\begin{proposition} \label{PropositionDefinition}
Given a weighted affine bundle $E \ra X$ (respectively a weighted vector bundle), we have a natural map $X \ra B\Gan$ (respectively $X \ra B\Gln$) such that $E$ is the pull back of $V\Gan$ (respectively $V\Gln$). 
\end{proposition}

\begin{proof}
We prove the result for $V\Gan$, the result for $V\Gln$ is effectively the same, with the obvious modifications. 
Let us denote with $Isom$ the sheaf of isomorphisms of affine bundles $Isom_X(E, V \times X)$ (respectively, the sheaf of isomorphisms of weighted vector bundles). 
By straightforward application of the definitions, it can be seen that $Isom$ is a principal $\Gan$ bundle over $X$ and $Isom \times_{X} E \cong V \times_{\{*\}} Isom$.

In particular we get the Cartesian diagram 
$$\begin{tikzcd}
Isom \arrow[r] \arrow[d] & \{*\} \arrow[d] \\
X \arrow[r] & B\Gan.
\end{tikzcd}$$

Then we can fit the spaces above in the following commutative cube

$$\begin{tikzcd}[sep = .4cm]
Isom \times_{X} E \arrow[rr]\arrow[dd]\arrow[dr]
&& V  \arrow[dd]\arrow[dr]
\\
& E && V\Gan \arrow[dd]\arrow[from=ll,dashed,crossing over]\\
Isom \arrow[rr]\arrow[dr]&&  \{*\}\arrow[dr]\\
& X \arrow[rr]\arrow[from=uu, crossing over] && B\Gan
\end{tikzcd}$$
 where the bottom, back and sides squares are fiber squares.

Notice that $Isom \times_{X} E \to E$ is a principal $\Gan$ bundle. Moreover, the action of $\Gan$ on $Isom$ gives a $\Gan$-equivariant map $Isom \times_{X} E \to V$ via the identification of $Isom \times_{X} E$ with $V \times Isom$. This gives us a map or quotient stacks $E\to V\Gan$ which makes the top of the cube Cartesian.

It follows that $$\begin{tikzcd}
E \arrow[r] \arrow[d] & V\Gan \arrow[d] \\
X \arrow[r] & B\Gan
\end{tikzcd}$$
is a fiber square, as needed.

\end{proof}

\begin{corollary} \label{cubediagram}
Let $E \ra X$ be a weighted affine bundle, with corresponding map $X \to B\Gan$ . Then $E'$, the pullback of $E$ via the map $X'=X \times_{B\Gan}\Gln \ra X$, is a weighted \emph{vector} bundle.
\end{corollary}

\begin{proof}
Consider the following diagram

$$\begin{tikzcd}[sep = .4cm]
E' \arrow[rr]\arrow[dd]\arrow[dr]
&& X' \arrow[dd]\arrow[dr]
\\
&V\Gln && B\Gln \arrow[dd]\arrow[from=ll,crossing over]\\
E \arrow[rr]\arrow[dr]&& X\arrow[dr]\\
& V\Gan \arrow[rr]\arrow[from=uu, crossing over] && B\Gan
\end{tikzcd}$$
The back and right squares are Cartesian by construction.
By Lemma \ref{GroupLemma2ElectricBoogaloo} and Proposition \ref{PropositionDefinition} we have that the front and bottom squares are Cartesian. 
Any such cube with these sides Cartesian is Cartesian, in particular the top square. 
Since $E'\to X'$ is the pullback of the vector bundle $V\Gln \ra B\Gln$, it is a vector bundle as desired.
\end{proof}

\begin{lemma} \label{isomorphismofChowrings}
Let  $X' \ra X$ be as in Corollary \ref{cubediagram}, then the pullback map of Chow rings $A^*(X) \ra A^*(X')$ is an isomorphism. 
\end{lemma}

\begin{proof}
By Corollary \ref{Uanbundle},  $X' \xrightarrow{\phi} X$ is a $\Uan$-bundle and $\Uan$ is a unipotent group. In particular, $\Uan$ is a successive extension of the additive group $\G_a$ by its self and being a $\Uan$ bundle is equivalent to being a succession of affine bundles, hence by \cite[\href{https://stacks.math.columbia.edu/tag/0GUB}{Tag 0GUB}]{stacks-project} we obtain an isomorphism of Chow rings $\phi^*: A^*(X) \ra A^*(X')$.
\end{proof}

\begin{theorem}[The splitting principle] \label{splitting principle}
Let $T$ be the standard maximal torus in $\Gln$. Then the map $X'' \ra X$ in the following diagram $$\begin{tikzcd}
X'' \arrow[r] \arrow[d] & BT \arrow[d]\\
X' \arrow[r] \arrow[d] & B\Gln \arrow[d] \\
X \arrow[r] & B\Gan
\end{tikzcd}$$ induces an injection of Chow rings $A^*(X) \hookrightarrow A^*(X'')$ via pullback. 
\end{theorem}

\begin{proof}
By the argument in the proof of \cite[Theorem 2.13]{totaro} we have an injection $A^*(X') \hookrightarrow A^*(X'')$. By composing with the isomorphism in Lemma \ref{isomorphismofChowrings}, we have the desired map. 
\end{proof}

\section{The Gysin homomorphism induced by a weighted blow-up}

The goal for this section is to prove Theorem \ref{Gysin homomorphism for weighted blow-up}, which replaces the excess bundle formula in the case of weighted blow-ups. 

The strategy for the proof is to reduce to the special case of the weighted blow-up of $BT$ in $[\A^d/T]$, which will be computed in Section \ref{The special case}. 

The reduction to the special case is performed in two steps: first we reduce to the case of the blow-up of an affine space (Section \ref{specializationtonormalcone}), and then we apply the splitting principle Theorem \ref{splitting principle}. 

Some caution is needed when defining $f^!$, as we don't always have the needed Cartesian diagram. In Section \ref{sectionnotation} we address the issue as well as setting some notation for the rest of the paper.

\subsection{Notation}\label{sectionnotation}

Let $\tilde Y \ra Y$ be the weighted blow-up of $Y$ centered at $X$, and let $\tilde X$ be the exceptional divisor. 

As observed in \cite[Remark 3.2.10]{quek-rydh-weighted-blow-up} the commutative square is not always Cartesian
$$\begin{tikzcd}
\tilde X \arrow[r] \arrow[d] & \tilde Y \arrow[d] \\
X \arrow[r] & Y
\end{tikzcd} $$
and when defining $f^!$ we have to make sure to define it with respect to the fiber square $$\begin{tikzcd}
X \times_{Y} \tilde Y \arrow[r] \arrow[d] & \tilde Y \arrow[d] \\
X \arrow[r] & Y
\end{tikzcd}.$$

Moreover we have $\tilde X=(X \times_{Y} \tilde Y)_{red}$ and the diagram below commutes $$\begin{tikzcd}
\tilde  X \arrow[r] \arrow[dr,"g"] & X \times_{Y} \tilde Y \arrow[r,"j"]\arrow[d,"h"]&
\tilde Y \arrow[d,"f"]\\ &
X \arrow[r,"i"]& Y
\end{tikzcd}$$

When looking at Chow rings though, we have a natural isomorphism $(red)_*: A^*(\tilde X) \ra A^*(X \times_{Y} \tilde Y)$ induced by the reduction map $red: \tilde X \ra X \times_{Y} \tilde Y$. 

In particular, it makes sense to talk about $f^!$ as the composition of $(red)_*^{-1} \circ f^!$.  Throughout the rest of the paper, we will refer to it simply as $f^!$. 

\begin{lemma} \label{structure of the map} The map $f^!: A^*(X) \ra A^*(\tilde X)$ is of the form $f^!(\alpha)=g^*(\alpha) \cdot \gamma$ for some element $\gamma \in A^*(\tilde X)$. 
\end{lemma}

\begin{proof}
In a similar fashion to what we did in Proposition \ref{remarks}, we will prove the statement by passing through algebraic spaces.

Precisely, let $\tilde X_{U}=(\mathcal{N}_XY \smallsetminus 0 ) \times U  / \G_m$ with $U$ open as in Definition \ref{definition of G-equivariant Chow groups} inducing isomorphisms of Chow groups for $\tilde X$ of the appropriate degree. 
In fact, if $U$ is chosen large enough we also get the algebraic space $\tilde Y_U$ with analogous induced isomorphisms of Chow groups for $\tilde Y$. 

For the appropriate degrees, the induced maps $f_U: \tilde Y_U \ra Y$ with  $(\tilde y, u) \mapsto f(\tilde y)$ and $g_U: \tilde X_U \ra X_U$ with $(\tilde x, u ) \mapsto g(\tilde x)$ will themselves induce group homomorphisms $f^!_U: A^*(X) \ra A^*(X_U)$ and $g^*_U: A^*(X) \ra A^*(X_U)$.

Now $g_U: \tilde X_U \ra X$ is a smooth map and by \cite[Theorem 17.4.2]{fulton}  we have isomorphisms $$ A^p(\tilde X_U)\cong A^p(\tilde X_U \xrightarrow{id} \tilde X_U) \xrightarrow{\cdot [g_U]} A^{p-d}(\tilde X_U \ra X). $$
In particular, for the degrees on which $A^p(\tilde X_U) = A^{p}_{\G_m}(\mathcal{N}_XY \smallsetminus 0)\cong A^{p}(\tilde X)$, we have that $f_U^!=\gamma_U \cdot [g_U]=\gamma_U\cdot g^*_U$ for some $ \gamma_U \in A^{p}(X_U)$, is equivalent to say $f^!(\alpha)=\gamma_U \cdot g^*(\alpha)$ for some element $\gamma_U \in A^{p}(\tilde X)$.

Since the elements $\gamma_U$ must agree whenever $U$ has high enough dimension, they must coincide. Hence there exists a unique element $\gamma \in A^*(\tilde X)$ such that $f^!(\alpha)=\gamma \cdot g^*(\alpha)$.

 \end{proof}

\subsection{Specialization to the weighted normal cone}\label{specializationtonormalcone}

Analogously to \cite[Section 5.2]{fulton}, Quek and Rydh in \cite[Section 4.3]{quek-rydh-weighted-blow-up} construct a deformation to the weighted normal cone of $X$ in $Y$, which is a weighted affine bundle in our case. We will be using their construction to reduce our argument to the case where $Y$ a weighted affine bundle over $X$. 
A similar construction can be found in \cite[Section 2.3]{mustata-mustata}. 

Note that when given a weighted embedding that defines a weighted blow-up of smooth varieties, the weighted normal cone is an affine bundle, which we will denote $\mathcal{N}_XY$.

\begin{theorem} \label{specialization}
Let $X$, $Y$, $\tilde{X}$, $\tilde{Y}$ be as usual. Let $N=\mathcal{N}_X Y$ the weighted normal affine bundle of $X$ in $Y$ and $f_N:\tilde{N}\to N$ the weighted blow-up with the same weights as $f:\tilde{Y}\to Y$. Then the induced maps $f^!: A^*(X) \ra A^*(\tilde X)$ and $f_N^!: A^*(X) \ra A^*(\tilde X)$ coincide.
\end{theorem}

\begin{proof}

Let $M^o$ be the deformation to the weighted normal cone as defined in \cite[Definition 4.3.3]{quek-rydh-weighted-blow-up} and let $\tilde{M^o}$ be the weighted blow-up of $X \times \A^1$ in $M^o$  with the same weights as $f$, i.e.,the weighted blow-up induced by the weighted embedding in \cite[Proposition 4.3.3]{quek-rydh-weighted-blow-up}. Let $M_t$ and $\tilde{M_t}$ respectively, be the fibers over $t$. Let $Z=X \times_Y \tilde Y$, note that $Z=X \times_N \tilde N$. Then we have the following diagram. 

$$
\begin{tikzcd}[sep = .4cm]
&\tilde{X}\times\A^1 \arrow[rr]\arrow[ddrr, "\textstyle Z" description] && Z\times \A^1 \arrow[rr]\arrow[dd]&& \tilde{M^o} \arrow[dd,"f_M"]\\
\tilde{X} \arrow[rr]\arrow[ddrr]\arrow[ur] && \contour{white}{$Z$} \arrow[rr]\arrow[dd]\arrow[ru] && \tilde{M_t} \arrow[ru] \\
&&& X \times\A^1 \arrow[rr]&& M^o\\
&& X \arrow[rr]\arrow[ru] && M_t \arrow[ru]\arrow[from=uu, crossing over]
\end{tikzcd}
$$

By looking at the composition $X \to M_t\to M^o$ in the subdiagram

$$\begin{tikzcd}
Z \arrow[r] \arrow[d] & \tilde{M_t} \arrow[r] \arrow[d] & \tilde{M^o} \arrow[d, "f_M"] \\
X \arrow[r] & M_t \arrow[r] & M^o
\end{tikzcd} $$
we see that for $t\neq 0$ we have that $f_M^!:A^*(X)\to A^*(Z)$ is precisely $f^!$, and for $t=0$ it is precisely $f_N^!$.
Now looking at the composition $X\to X\times \A^1\to M^o$ in the subdiagram
$$\begin{tikzcd}
Z \arrow[r] \arrow[d] & Z\times \A^1 \arrow[r] \arrow[d] & \tilde{M^o} \arrow[d, "f_M"] \\
X \arrow[r] & X\times \A^1 \arrow[r] & M^o
\end{tikzcd} $$
we see that $f_M^!:A^*(X)\to A^*(Z)$  is the same for all $t$.

\end{proof}

\subsection{The special case of $[\A^d/T]$}\label{The special case}

Let us now study the particular case of a point in the affine space over the torus.

$$\begin{tikzcd}
\mathcal{P}(a_1,...a_d) \arrow[r] \arrow[dr,"g"] & 0 \times_{\A^d} Bl_{a_1,...,a_d}\A^d \arrow[r,"j"]\arrow[d,"h"]&
Bl_{a_1,...,a_d}\A^d \arrow[d,"f"]\\ &
0 \arrow[r,"i"]&
\A^d
\end{tikzcd}$$

In order to explicitly give a formula for $f^!$ we need presentations for the equivariant Chow rings $A^*_T(-)$.  

Now $A^*_T(0) \cong A^*_T(\A^d) \cong \Z[x_1,...,x_d]$. Details about $A^*_T(0)$ can be found in \cite{EG} and in \cite{Iwanari} for equivariant Chow rings of toric stacks.

Let us first observe that, since $\mathcal{P}(a_1,...,a_d)$ is the reduction of $0 \times_{\A^d} Bl_{a_1,...,a_d}\A^d$, there is an isomorphism of Chow rings $A^*_T(0 \times_{\A^d} Bl_{a_1,...,a_d}\A^d) \cong A^*_T(\mathcal{P}(a_1,...,a_d))$.

Moreover $Bl_{a_1,...,a_d}\A^d$ is a line bundle over $\mathcal{P}(a_1,...,a_d)$, in fact it is the total space of $\mathcal{O}_{\mathcal{P}(a_1,...,a_d)}(-1)$, and we have the isomorphism $A^*_T(Bl_{a_1,...,a_d}\A^d) \cong A^*_T(\mathcal{P}(a_1,...,a_d))$

We are left with computing $A^*_T(\mathcal{P}(a_1,...,a_d))$.

\begin{lemma}{$A^*_{T}(\mathcal{P}(a_1,...,a_d)) \cong \frac{\Z[x_1,...,x_d,t]}{P(t)}$} where $P(t):=\prod_{i=1}^d (x_i+ta_i)$
\end{lemma}

\begin{proof} By construction $\mathcal{P}(a_1,...,a_d)=[\A^d \smallsetminus 0/\G_m]$ and the actions of $\G_m$ and $T$ on $\A^d \smallsetminus 0$ commute. In particular $A^*_T(\mathcal{P}(a_1,...,a_d)) \cong A^*_T([A^d \smallsetminus 0/\G_m]) \cong A^*_{T \times \G_m}(\A^d \smallsetminus 0)$.

Similarly to the computation above, $A^*_{T \times \G_m}(0) \cong A^*_{T \times \G_m}(\A^d) \cong \Z[x_1,...,x_d,t]$ where $x_1,...,x_d$ are given by the action of $T$ and $t$ is given by the action of $\G_m$. Finally, the image of the first map in the localization sequence: $$A^*_{T \times \G_m}(0) \ra A^*_{T \times \G_m}(\A^d) \ra A^*_{T \times \G_m}(\A^d \smallsetminus 0) \ra 0$$ is generated by $P(t):=\prod_{i=1}^d (x_i+ta_i)$. Indeed the top Chern class of the $T \times \G_m$-equivariant bundle splits along each component of $\A^d$. On the $i$-th component of $\A^d$ the $i$-th component of $T$ acts with weight $1$ and the other components of $T$ act with weight $0$, while $\G_m$ acts with weight $a_i$.

Therefore $A^*_{T}(\mathcal{P}(a_1,...,a_d))$ and $A^*_{T}(Bl_{a_1,...,a_d}\A^d)$ are both isomorphic to $\frac{\Z[x_1,...,x_d,t]}{P(t)}$.\end{proof}

\begin{theorem} \label{computationforT} 
Let $f: [Bl_{a_1,...,a_d}\A^d/T] \ra [\A^d/T]$ be the blow-up of $[0/T]$ in $[\A^d/T]$ with weights $a_1,...,a_d$. Then
$$f^!(1)=\left( \frac{c_{top}^{\G_m}([\A^n/T])(t)-c_{top}^{\G_m}(\A^n/T)(0)}{t} \right)$$
\end{theorem}

\begin{proof}
By \cite[Theorem 3.12]{vistoli}
$f^!$ must satisfy $f^*i_*=j_*f^!$, making the following diagram commute. $$
\begin{tikzcd}
A^*(\mathcal{P}(a_1,...,a_d)) \arrow[r,"j_*"]& 
A^*(Bl_{a_1,...,a_d}\A^d )\\ 
A^*(0) \arrow[r,"i_*"]\arrow[u, "f^!(1)\cdot g^*"]&
A^*(\A^d) \arrow[u, "f^*"]
\end{tikzcd} $$
Now $i_*: A^*_T(0) \ra A^*_T(\A^d)$ is just the multiplication by the top equivariant Chern class of the bundle $\A^d$ over $0$. Specifically $i_*(\alpha)=\alpha \cdot x_1...x_d$. \\
Similarly, $j_*$ is the image of the top equivariant Chern class of the bundle $Bl_{a_1,...,a_d}\A^d$ over $\mathcal{P}(a_1,...,a_d)$, so we have that $j_*$ is multiplication by $-t$.\\
Therefore we must have, $$f^*i_*(\alpha)=\alpha \cdot x_1\cdots x_d=\alpha \cdot c_{top}^{\G_m}([\A^n/T])(0)= f^!(1) \alpha (-t)=j_*f^!(\alpha)$$ 

and since $t$ is not a zero divisor in $A^*_T(\mathcal{P}(a_1,...,a_d))$, we must have $$f^!(1)=\frac{-c_{top}^{\G_m}([\A^n/T])(0)}{t}=\frac{c_{top}^{\G_m}([\A^n/T])(t)-c_{top}^{\G_m}([\A^n/T])(0)}{t},$$ as needed. \end{proof}

\subsection{A fomula for the Gysin homomorphism}

\begin{theorem}\label{Gysin homomorphism for weighted blow-up}
Let $X$, $Y$, $\tilde{X}$, $\tilde{Y}, f$ be as usual. Let us identify $A^*(\tilde X) \cong A^*(X)[t]/P(t)$ with $P(t)=c_{top}^{\G_m}(\mathcal{N}_XY)(t)$.
Then we have the following formula for the Gysin homomorphism $f^!:A^*(X) \ra A^*(\tilde X)$. $$
 f^!(\alpha)=\frac{P(t)-P(0)}{t}\alpha.$$
\end{theorem}

\begin{proof}
With the presentation of $A^*(\tilde X)$ above, the map $g^*$ is the natural inclusion of $A^*(X)$ in $A^*(X)[t]/P(t)$ and, by Lemma \ref{structure of the map} we only need to show $$f^!(1)=\frac{P(t)-P(0)}{t}.$$

By Theorem \ref{specialization} we can assume that $Y$ is a weighted affine bundle over $X$. By the splitting principle in Theorem \ref{splitting principle} it is enough to prove the equality for the pullback $X''$. 
Since blow-ups commute with base change, the blow-up $f'': \tilde Y'' \ra Y''$ sits in the commutative diagram 
$$\begin{tikzcd}
\tilde{X''} \arrow[rr, "\tilde \phi"]\arrow[dd]\arrow[dr]
&& \left[\mathcal{P}(a_1,...,a_n)\right/T] \arrow[dd]\arrow[dr]
\\
&\tilde Y ''&& \left[Bl_{a_1,...,a_n}\A^n/T\right] \arrow[dd]\arrow[from=ll,crossing over]\\  X''\arrow[rr, "\phi"]\arrow[dr]&& BT \arrow[dr]\\
& Y'' \arrow[rr]\arrow[from=uu, crossing over] && \left[\mathbb{A}^n/T\right]
\end{tikzcd}$$

induces a commutative diagram of Chow groups

$$\begin{tikzcd}
A^*(\tilde{X''}) \arrow[from=dd, "(f'')^!"]\arrow[dr]
&& A^*(\left[\mathcal{P}(a_1,...,a_n)\right/T]) \arrow[ll, "\tilde \phi^*"]\arrow[from=dd]\arrow[dr]
\\
&A^*(\tilde Y '') && A^*(\left[Bl_{a_1,...,a_n}\A^n\right/T]) \arrow[from=dd]\arrow[ll,crossing over]\\  A^*(X'')\arrow[from=rr, "\phi^*"]\arrow[dr]&& A^*(BT) \arrow[dr]\\
& A^*(Y'') \arrow[from=rr]\arrow[uu, crossing over] && A^*(\left[\mathbb{A}^n/T\right]).
\end{tikzcd}$$

Since equivariant Chern classes commute with pullbacks and $Y''$ is a vector bundle over $X''$, by Theorem \ref{computationforT} the following holds $$(f'')^!(1)=(f'')^!(\phi^*(1))=\tilde\phi ^* \left( \frac{c_{top}^{\G_m}([\A^n/T])(t)-c_{top}^{\G_m}(\A^n/T)(0)}{t} \right) =$$ $$=\left( \frac{c_{top}^{\G_m}(\tilde \phi ^*[\A^n/T])(t)-c_{top}^{\G_m}(\tilde \phi ^*[\A^n/T])(0)}{t} \right) =\left( \frac{c_{top}^{\G_m}(Y'')(t)-c_{top}^{\G_m}(Y'')(0)}{t} \right)$$ which is the desired difference quotient. 

\end{proof}

\section{The Chow ring of a weighted blow-up}

In this section we generalize the key sequence in \cite[Proposition 6.7(e)]{fulton} and then use it to compute a formula for the Chow ring of a weighted blow-up.

Let us recall the notation. Let $f:\tilde Y \ra Y$ the weighted blow-up of $Y$ at $X$. Let $\tilde X$ be the exceptional divisor. Then we have the commutative diagram
$$\begin{tikzcd}
\tilde X \arrow[r, "j"] \arrow[d, "g"] & \tilde Y \arrow[d, "f"] \\
X \arrow[r, "i"] & Y.
\end{tikzcd}$$

\subsection{The Grothendieck sequence}
\begin{theorem}[Key sequence]\label{sequence}
Let $X$, $Y$, $\tilde{X}$, $\tilde{Y}$, $f$, $i$, $j$ be as above, then we have the following exact sequence of Chow groups
$$A^*(X) \xrightarrow{(f^!, - i_*)} A^*(\tilde X) \oplus A^*(Y) \xrightarrow{j_*+f^*} A^*(\tilde Y) \ra 0.$$
Further, if we use rational coefficients, then this becomes a split short exact sequence with $g_*$ left inverse to $(f^!, -i_*)$.
$$0\to A^*(X,\mathbb{Q}) \xrightarrow{(f^!, -i_*)} A^*(\tilde X,\mathbb{Q}) \oplus A^*(Y,\mathbb{Q}) \xrightarrow{j_*+f^*} A^*(\tilde Y,\mathbb{Q}) \ra 0.$$
\end{theorem}

\begin{proof}
 To prove exactness let us look at the double complex of higher Chow groups $$\begin{tikzcd}
\cdots A^*(U, 1) \arrow[r, "\tilde \delta_1"] & A^*(\tilde{X}) \arrow[r, "j_*"] & A^*(\tilde{Y}) \arrow[r]  & A^*(U) \arrow[r] & 0\\ \cdots A^*(U, 1) \arrow[u, equal]\arrow[r, "\delta_1"] &
A^*(X) \arrow[r]\arrow[u, "f^!"] & A^*(Y) \arrow[r]\arrow[u,"f^*"] & A^*(U) \arrow[r]\arrow[u,equal] & 0
\end{tikzcd}$$ where $U \cong \tilde Y \smallsetminus \tilde X \cong Y \smallsetminus X$.

Since both of the complexes are exact, the total complex

$$\cdots A^*(U, 1) \oplus A^*(X) \ra A^*(\tilde X) \oplus A^*(Y) \ra A^*(\tilde Y) \oplus A^*(U) \ra A^*(U)  \ra 0$$ is also exact.

Let us prove that the map $A^*(\tilde X) \oplus A^*(Y) \xrightarrow{j_*+f^*} A^*(\tilde Y)$ is surjective.  
Let $\alpha$ be any cycle in $A^*(\tilde Y)$, $\bar{\alpha}$ the restriction of $\alpha$ to $A^*(U)$, and $\gamma\in A^*(Y)$ any cycle that restricts to $\bar{\alpha}$ in $A^*(U)$. By commutativity, $\alpha-f^*(\gamma)$ restricts to 0 in $A^*(U)$ and must be in the image of $j_*$. So $\alpha$ is in the image of $j_*+f^*$.
 Therefore the complex 
$$\cdots A^*(U, 1) \oplus A^*(X) \ra A^*(\tilde X) \oplus A^*(Y) \ra A^*(\tilde Y) \ra 0$$ is still exact. 

Moreover, the image of the map $A^*(U,1) \oplus A^*(X) \xrightarrow{(-\tilde \delta_1+f^!, -i_*)} A^*(\tilde X) \oplus A^*(Y)$ coincides with the one of $A^*(X) \xrightarrow{(f^!, -i_*)} A^*(\tilde X) \oplus A^*(Y)$. Indeed, let $(\tilde x, y )=(-\tilde \delta_1(u)+f^!(x), -i_*(x))$ and let $x'=x+\delta_1(u)$. Then $f^!(x')=f^!(x)-f^!(\delta_1(u))=f^!(x)-\tilde \delta_1(u)=\tilde x$ and $i_*(x')=i_*(x)-i_*(\delta_1(u))=i_*(x)=y$.

It follows that $ker(j_*+f^*)=Im(f^!,-i_*)$ and that 

$$A^*(X) \xrightarrow{(f^!, - i_*)} A^*(\tilde X) \oplus A^*(Y) \xrightarrow{j_*+f^*} A^*(\tilde Y) \ra 0$$
is exact.

Lastly, if we use rational coefficients then there is a left inverse of $(f^!, -i^*)$ given by $(\alpha, \beta) \mapsto g_*(\alpha)$.
Indeed, let $x \in A^*(X)$, then $g_*(f^!(x))=g_*(\delta g^*(x))=g_*(\delta) x$ with $\delta$ the difference quotient as in Theorem \ref{Gysin homomorphism for weighted blow-up}. Now $\delta$ is a degree $n-1$ polynomial in $t$, of which only the leading term $a_1 \cdots a_n t^{n-1}$ will survive the pushforward. We only need to show that $g_*(t^{n-1})=\frac{1}{a_1 \cdots a_n}$. 
It is enough to verify this when $X$ is a point. Notice that $a_it=x_i$, where the $x_i$ are the usual coordinate divisors, so $a_1\cdots a_{n-1}t^{n-1}=x_1\cdots x_{n-1}$ which is a $B\mu_{a_n}$ and so pushes forward to $\frac{1}{a_n}$, thus $g_*(t^{n-1})=\frac{1}{a_1\cdots a_n}$
\end{proof}

\begin{example} \label{non exactness} To see why the sequence with integer coefficients is not short exact,
let us consider $X$ an elliptic curve in $Y=\p^2$. Let $\tilde Y$ be the blow-up of $Y$ at $X$ with weight $2$. 
Let $P, Q \in X$ be distinct points of order $2$ and consider the difference $[P]-[Q] \in A^*(X)$. When pushed forward to $A^*(Y)$ via $i_*$, all points are rationally equivalent, hence $i_*([P]-[Q])=0$. On the other hand, $f^!$ is multiplication by $2$, so $f^!([P])=f^!([Q])=0$. But $[P] -[Q]$ is non zero, so $(f^!,-i_*)$ is not injective.
\end{example}

\begin{remark}
    Note that, when looking at the double complex in the proof of Theorem \ref{sequence} and taking the total complex, one defines a long exact sequence of higher Chow groups. The isomorphisms of higher Chow groups $$\alpha_i: A^*(\tilde Y \smallsetminus \tilde X, i) \ra A^*(Y \smallsetminus X, i)$$ allow us to delete $A^*(\tilde Y \smallsetminus \tilde X)$ and $A^*(Y \smallsetminus X)$ from the complex via a diagram chase analogous to the ones in the proof. 
    Then we can obtain the following long exact sequence 
    $$\cdots \ra A^*(X,i) \ra A^*(\tilde X, i) \oplus A^*(Y,i) \ra A^*(\tilde Y,i) \ra A^*(X, i-1) \ra \cdots$$
\end{remark}

\subsection{The Chow ring of a weighted blow-up}

\begin{theorem}[Chow ring of a weighted blow-up] \label{main theorem 4}Chow ring of a weighted blow-up] If $\tilde{Y}\to Y$ is a weighted blowup of $Y$ at a closed subvariety $X$, then the Chow ring $A^*(\tilde{Y})$ is isomorphic \emph{as a group} to the quotient $$A^*(\tilde Y)\cong\frac{(A^*(X)[t])\cdot t \oplus A^*(Y)}{( (({P(t)-P(0)})\alpha,-i_*(\alpha)), \forall\alpha\in A^*(X))}$$ with $P(t)=c_{top}^{\G_m}(\mathcal{N}_XY)(t)$ and $[\tilde X]=-t$.

The multiplicative structure on $A^*(\tilde Y)$ is induced by the multiplicative structures on $A^*(X)$ and $ A^*(Y)$ and by the pullback map in the following way $$(0,\beta) \cdot (t,0)=(i^*(\beta)t,0).$$

Equivalently $A^*(\tilde Y)$ can be expressed as a quotient of the fiber product 
$$ \frac{A^*(Y)\times_{A^*(X)} A^*(X)[t]}{((i_*\alpha, P(t)\alpha) \ \forall \alpha \in A^*(X))} $$ with $i^*: A^*(Y) \ra A^*(X)$ and $A^*(X)[t] \ra A^*(X)$ given by evaluating $t$ at $0$. 

\end{theorem}

\begin{proof}
The exact sequence in Theorem \ref{sequence} gives us an isomorphism of \emph{groups}
$$A^*(\tilde Y)\cong\frac{A^*(\tilde X) \oplus A^*(Y)}{((f^!(\alpha),-i_*(\alpha)), \forall\alpha\in A^*(X))}.$$
If we use Theorem \ref{WeighedProjectiveBundleFormula} to rewrite $A^*(\tilde{X})$ and also add an explicit factor of $[\tilde{X}]$ to represent how $A^*(\tilde{X})$ is mapped into $A^*(\tilde{Y})$, then as a \emph{group} we can rewrite $A^*(\tilde{Y})$ as
$$\frac{((A^*(X)[t])\cdot[\tilde{X}]) \oplus A^*(Y)}{((c_{top}^{\G_m}(\mathcal{N}_XY)(t)[\tilde{X}],0), (f^!(\alpha)[\tilde{X}],-i^*(\alpha)) \forall\alpha\in A^*(X))}.$$
Notice that $t[\tilde{X}]=-[\tilde{X}]^2$ (since $t$ is the class of $\mathcal{O}_{\tilde X}(1))$ so that $[\tilde{X}]$ can be replaced by $-t$.

Now we need to determine the ring structure. Since much of the ring structure is inherited from that of $A^*(Y)$ and $A^*(\tilde{X})$, what remains is just to determine how to multiply elements coming from $A^*(Y)$ with those coming from $A^*(\tilde{X})$. Consider the usual commutative square 
$$\begin{tikzcd}
\tilde X \arrow[r] \arrow[d] & \tilde Y \arrow[d] \\
X \arrow[r] & Y.
\end{tikzcd}$$
Intersecting some class $\beta\in A^*(Y)$ with $\tilde{X}$ amounts to pulling it back to $A^*(\tilde{X})$. By commutativity we have $g^*(i^*(\beta))=j^*(f^*(\beta))$ and by pushforward we obtain, $(0,\beta) \times (t,0)=(i^*(\beta)t,0)$.

Finally, notice also that $c_{top}^{\G_m}(\mathcal{N}_XY)(t)[\tilde{X}]$ is now redundant, as for $\alpha=1$ we have 
$$(t)(f^!(\alpha)(-t)-i_*(\alpha)) = (t)(\frac{c_{top}^{\G_m}(\mathcal{N}_XY)(t)-c_{top}(\mathcal{N}_XY)}{t}(-t)-i_*(1))$$
$$=t(-c_{top}^{\G_m}(\mathcal{N}_XY)(t)+c_{top}(\mathcal{N}_XY)-i_*(1))=(-t)\cdot c_{top}^{\G_m}(\mathcal{N}_XY)(t).$$
Where the last equality comes from $t\cdot i_*(1)=g^*(i^*(i_*(1)))=c_{top}(\mathcal{N}_XY)$.

Putting everything together, we have that the $A^*(\tilde{Y})$ is the group
$$A^*(\tilde Y)\cong\frac{(A^*(X)[t])\cdot t \oplus A^*(Y)}{( (({P(t)-P(0)})\alpha,-i_*(\alpha)), \forall\alpha\in A^*(X))}$$ with the desired multiplication.

\end{proof}

\begin{corollary}\label{Keel}
If $i^*:A^*(Y)\to A^*(X)$ is surjective, then this formula simplifies to resemble a formula of Keel \cite[Theorem 1, page 571]{Keel} 
$$A^*(\tilde{Y})\cong \frac{A^*(Y)[t]}{(t\cdot ker(i^*), Q(t))}$$
where $Q(t)=c_{top}^{\G_m}(\mathcal{N}_XY)(t)-c_{top}^{\G_m}(\mathcal{N}_XY)(0)+[X].$
\end{corollary}
\begin{proof}
Let $i^*$ be surjective. The fact any $\alpha\in A^*(X)$ can only appear multiplied by $t$ combined with the terms  $t\cdot(\beta-i^*(\beta)) \forall\beta\in A^*(Y)$ means we can identify every $\alpha\in A^*(X)$ with any $\beta$ such that $i^*(\beta)=\alpha$.
This identification allows us to suppress $A(X)$ from our presentation, and reduces the terms $t\cdot(\beta-i^*(\beta)) \forall\beta\in A^*(Y)$ to $t\cdot ker(i^*)$.

Finally, $t\cdot f^!(\alpha)+i_*(\alpha)=(c_{top}^{\G_m}(\mathcal{N}_XY)(t)-c_{top}(\mathcal{N}_XY))\alpha+i_*(\alpha)$ and $i_*(\alpha)=i_*(i^*(\beta))=[X]\cdot\beta$ for some $\beta\in A^*(Y)$.
So $t\cdot f^!+i_*$ is multiplication by $(c_{top}^{\G_m}(\mathcal{N}_XY)(t)-c_{top}(\mathcal{N}_XY)+[X])$ which is precisely the $Q(t)$ desired.
\end{proof}

\subsection{An example: the Chow ring of $\bar{\mathcal{M}}_{1,2}$}

The Chow ring $A^*(\bar{\mathcal{M}}_{1,2})$ has been computed in \cite{dilorenzo-pernice-vistoli} and in \cite{inchiostro}.
The latter uses the construction of $\bar{\mathcal{M}}_{1,2}$ as the weighted blow-up of $\mathcal{P}(2,3,4)$. We give yet another computation of the ring, using the same blow-up construction. 

We start by recalling the following:

\begin{theorem}{\cite[Theorem 2.6]{inchiostro}}\label{M12}
There exists an isomorphism $\bar{\mathcal{M}}_{1,2} \cong Bl^{(4,6)}_Z\mathcal{P}(2,3,4)$, where $Bl^{(4,6)}_Z\mathcal{P}(2,3,4)$ is the weighted blow-up of the point $Z=[s^2:s^3:0]$ in $\mathcal{P}(2,3,4)$ with weights $(4,6)$. 
\end{theorem}

Given this, $A^*(\bar{\mathcal{M}}_{1,2})$ becomes a straightforward computation,

\begin{proposition}{\cite[Theorem 4.12]{inchiostro}}
$$A^*(\bar{\mathcal{M}}_{1,2})\cong \frac{\mathbb{Z}[y,t]}{(ty,24(t^2+y^2))}.$$
\end{proposition}
\begin{proof}
First, since $i:Z\to\mathcal{P}(2,3,4)$ is just the inclusion of a point we have that $i$ is surjective, and since $A^*(\mathcal{P}(2,3,4))\cong \frac{\Z[y]}{24y^3}$ we know the kernel is $(y)$.
By Corollary \ref{Keel} we then have
$$A^*(\bar{\mathcal{M}}_{1,2})\cong \frac{A^*(\mathcal{P}(2,3,4))[t]}{(ty,Q(t))}\cong \frac{\mathbb{Z}[y,t]}{(24y^3,ty,Q(t))}.$$
Where $Q(t)$ restricts to $c_{top}^{\G_m}(\mathcal{N}_ZY)$ and has constant term $[Z]$.
As $\mathcal{N}_ZY$ splits into trivial line bundles, we see $c_{top}^{\G_m}(\mathcal{N}_ZY)=(4t)(6t)$.
Writing $Z=V(x_3)\cap V(x_1^3-x_2^2)$ we see $[Z]=(4y)(6y)$, so $Q(t)=24t^2+24y^2$.
Lastly, the term $24y^3$ is now redundant and we have
$$A^*(\bar{\mathcal{M}}_{1,2})\cong \frac{\mathbb{Z}[y,t]}{(ty,24(t^2+y^2))}.$$
\end{proof}

\section{Generalization to quotient stacks}

Let us now consider the case of $\srY=[Y/G]$, where $Y$ is an algebraic space and $G$ is a linear algebraic group, hence it is possible to define the $G$-equivariant Chow ring $A^*_G(Y)$ as in \cite{EG}.

A weighted embedding of $\srX $ in $ \srY$ defines a weighted embedding of $X $ in $ Y$ via pullback and, since the quotient maps are smooth, we have $\tilde \srY \cong [\tilde Y/G]$ and $\tilde \srX \cong [\tilde X/G]$.

\begin{theorem}
    The theorems in Sections 3,5,6 hold for $f: \tilde \srY \ra \srY$ weighted blow-up of $\srY$ at $\srX$.
\end{theorem}

\begin{proof}
    Let us prove that $A^*(\tilde \srX) \cong \frac{A^*(\srX)[t]}{P(t)}$ as in Theorem \ref{WeighedProjectiveBundleFormula}, the proof for the other results will be almost identical. 

For any $p$ let $U$ as in Definition \ref{definition of G-equivariant Chow groups} of dimension high enough, such that $A^q(X_U) \cong A^q_G(X)\cong A^q(\srX)$ with $X_U:= X \times U /G$, up to degree $p$. 

    Since $X_U$ is an algebraic space, by Theorem \ref{WeighedProjectiveBundleFormula} $$A^*(\tilde X_U) \cong A^*(X_U)[t]/P_U(t).$$
    Note that $P_U(t)=c_{top}^{\G_m}(\mathcal{N}_{X_U}Y_U)$ is the pullback of $P(t)=c_{top}^{\G_m}(\mathcal{N}_{\srX}\srY)$, which is a finite degree polynomial. In particular for large enough $p$, $P_U(t)$ does not depend on $U$ and it is exactly $P(t)$. 
    
    Since $\tilde X_U$ is an open inside a vector bundle, we have isomorphisms $A^q(\tilde X_U) \cong A^q(\tilde \srX)$ up to degree $p$. 
    Since \ref{WeighedProjectiveBundleFormula} holds up to degree $p$, for any $p$ we have the desired isomorphism of Chow rings.

\end{proof}

\appendix
\section{Chern Class of Weighted blow-up}
\begin{center} by Dan Abramovich, Veronica Arena, and Stephen Obinna 
\end{center}

\subsection{The goal} Consider a smooth subvariety $X$ of a smooth variety $Y$, with blowup $\tilde Y$ and exceptional divisor, as in the following standard diagram:
$$\begin{tikzcd} \tilde X \ar[r,"j"]\ar[d,"g"] & \tilde Y \ar[d,"f"] \\
                   X \ar[r,"i"] & Y.\end{tikzcd}$$

In \cite[Theorem 15.4]{fulton} Fulton provides a formula for the total Chern class $c(\tilde Y):= c(T_{\tilde Y})$ in terms of the blowup data. The purpose of this note is to revisit that formula and generalize it to the case of a weighted blowup. Since smoothness is important in these considerations, our weighted blowups are always stack-theoretic.

\subsection{Setup and formula}
In our setup, $X$ and $Y$ are still smooth varieties, and $X$ is the support of a weighted center with weighted normal bundle $N$ of rank $d=\codim(X\subset Y)$. The weighted normal bundle is a weighted affine bundle with total Chern class we denote by $c(N)\in A^*(X)$ and total $\G_m$-equivariant Chern class $c^{\G_m}(N)= Q(t)\in A^*_{\G_m}(X) = A^*(X)[t]$, where $t$ is the equivariant parameter corresponding to the standard character of $\G_m$. In particular $Q(0) = c(N)$.

We recall from Theorem \ref{main theorem 4} in the main text that $$A^*(\tilde Y) = \left(A^*(Y) \oplus t A^*(\tilde X) \right)/ I,$$ where $$I = \left(i_*(\alpha) \oplus -(Q(t)-Q(0))\alpha \big| \alpha\in A^*(X)\right).$$

We denote by $$q: A^*(Y) \oplus t A^*(X)[t]\qquad  \to \qquad \left(A^*(Y) \oplus t A^*(\tilde X)\right)/I\ \  = \ \ A^*(\tilde Y)$$ the natural quotient map.

\begin{theorem} We have 
$$ \frac{c(\tilde Y)}{ f^* c(Y)} = q\left(\frac{(1-t) Q(t)}{Q(0)}\right). $$
\end{theorem}
We note that  the right-hand side is  of the form $q(1\oplus t\cdot R(t))$, with some $R(t) \in A^*(X)\llbracket t\rrbracket$.

The formula was proved for Chow groups with rational coefficients by Anca and Andrei Musta\c{t}a, see \cite[Proposition 2.12]{mustata-mustata}. Our proof in essence verifies that their arguments carry over integrally.

\subsection{Approach}
Our approach combines the equivariant methods used in the main text to study and compute Chow rings of weighted projective stack bundles and weighted blowups, combined with ideas in Aluffi's paper and lecture \cite{Aluffi-Chern,Aluffi-lecture}, especially the user-friendly presentation of the formula in Aluffi's lecture. While Aluffi provides a proof only for complete intersections, the methods of Theorem \ref{main theorem 4} allow us to reduce the general case to a situation where Aluffi's proof applies.

\subsubsection{The quotient class}
One first notes that the class $\frac{c(\tilde Y)}{ f^* c(Y)} $ appearing on the left-hand side has properties enabling flexible treatment:

\begin{lemma} The class $\frac{c(\tilde Y)}{ f^* c(Y)} $ is of the form  $q(1\oplus t\cdot R(t))$, with some $R(t) \in A^*(X)\llbracket t\rrbracket$, and is functorial for smooth morphisms $Y'\to Y$ and closed embeddings $Y'\to Y$ that meeet $X$ transversely.
\end{lemma}
\begin{proof}
    To see that it has this form, consider the localization sequence:
    $$A^*(\tilde{X})\to A^*(\tilde{Y})\to A^*(\tilde{Y}\smallsetminus\tilde{X})\to0$$
    Since $c(\tilde Y)$ and $f^* c(Y)$ must pull back to the same class on $A^*(\tilde{Y}\smallsetminus\tilde{X})$, and their ratio pulls back to one. In particular we see that $\frac{c(\tilde Y)}{f^*c(Y)}-1$ must be in the image of $A^*(\tilde{X})$, which means $\frac{c(\tilde Y)}{f^*c(Y)}$ is of the desired form. 

    To see functoriality, consider the diagram 
    $$\begin{tikzcd}
        \tilde{Y'} \arrow[r,"\tilde{h}"]\arrow[d,"f'"] & \tilde{Y} \arrow[d,"f"]\\
        Y' \arrow[r,"h"] & Y
    \end{tikzcd}$$

    We must show $\frac{c(\tilde{Y'})}{f'^*c(Y')}=\tilde{h}^*\frac{c(\tilde{Y})}{f^*c(Y)}$, but this is equivalent to $\frac{c(\tilde{Y'})}{\tilde{h}^*c(\tilde{Y})}=f'^*\frac{c(Y')}{h^*c(Y)}$. This is true when $h$ is smooth because the relative tangent bundle of $h$ is compatible with pullback under $f$, and true when $h$ is a closed embedding since the normal bundle of $h$ is compatible with pullback under $f$.
\end{proof}

\subsubsection{Degeneration to the weighted normal bundle}
Applying the lemma to the degeneration to the weighted normal bundle we obtain
\begin{lemma} It suffices to prove the theorem, namely to compute $R(t)$ and $\frac{c(\tilde Y)}{ f^* c(Y)} $, when $Y = \mathcal{N}_XY$.
\end{lemma}
\begin{proof}
    Recall the diagram from Theorem \ref{specialization}
    
    $$\begin{tikzcd}[sep = .4cm]
&\tilde{X}\times\A^1 \arrow[rr]\arrow[ddrr, "\textstyle Z" description] && Z\times \A^1 \arrow[rr]\arrow[dd]&& \tilde{M^o} \arrow[dd,"f_M"]\\
\tilde{X} \arrow[rr]\arrow[ddrr]\arrow[ur] && \contour{white}{$Z$} \arrow[rr]\arrow[dd]\arrow[ru] && \tilde{M_t} \arrow[ru] \\
&&& X \times\A^1 \arrow[rr]&& M^o\\
&& X \arrow[rr]\arrow[ru] && M_t \arrow[ru]\arrow[from=uu, crossing over]
\end{tikzcd}$$

Where $M_t\cong Y$ for $t\neq0$ and $M_0=\mathcal{N}_XY$.

By the previous lemma, the expression $\frac{c(\tilde M_t)}{ f^* c(M_t)} $ can be pulled back from $\tilde{M}^o$ along the embedding corresponding to $t$ and is determined by a class on $\tilde{X}$. However, neither $\tilde{X}$ nor $\tilde{M}^o$ depend on $t$ so it is enough to compute things when $t=0$, that is for $\mathcal{N}_XY$.

\end{proof}

\subsubsection{The universal case} 

 By Theorem \ref{splitting principle}, the homomorphism $A^*(B\Gan) \to A^*(BT)$ is injective. Therefore: 
\begin{lemma}
    It suffices to prove the theorem when $X = BT$ and $Y = [V/T]$. Equivalently, it suffices to prove the theorem $T$-equivariantly when $X$ is a point, the origin on $Y = \AA^d$
\end{lemma}
\begin{proof}
    Follows from functoriality and Theorem \ref{splitting principle} since the maps $X\to B\Gan$ and $BT\to B\Gan$ are smooth.
\end{proof}

\subsubsection{The toric case of affine space}
Finally, let $X$ be the origin of $Y = \AA^d$.
Let $A^*_T(0) \cong A^*_T(\mathbb A^d) \cong \mathbb Z [x_1, \dots x_d]$ and $A^*_T(\mathcal{P}(a_1,\dots a_d)) \cong A^*_T(Bl_{(a_1, \dots a_d)} \mathbb A^d) \cong \mathbb Z [x_1, \dots x_d, t]/(\prod (x_i+a_it)) $ as in Section \ref{The special case}.

\begin{proposition}
    We  have $ c^T(Y) = Q(0) $ and $ c^T(\tilde Y) = q((1-t) Q(t)) $.
\end{proposition}
\begin{proof}

This is essentially the same argument as \cite[Theorem 4.2]{Aluffi-Chern-singular}.

Let $D=\Sigma \tilde X_i +\tilde{X}$ be the sum of all the irreducible toric divisors on $\tilde Y$. By repeating the argument in \cite[Proposition p.87]{Fulton-toric}, we have the exact sequence $$0 \ra \Omega_{\tilde Y}^1 \ra \Omega^1_{\tilde Y}(log D) \ra \left(\bigoplus_{i=1}^d \mathcal{O}_{X_i} \right) \oplus \mathcal{O}_{\tilde X} \ra 0$$
and $\Omega^1_{\tilde Y}(log D)$ is trivial. By Whitney's formula, $$c^T(\Omega^1_{\tilde Y}) = \frac{1}{c^T(\mathcal{O}_{\tilde X})c^T(\oplus \mathcal{O}_{\tilde X_i})}=(1+t) \prod_i (1-a_it-x_i).$$
By taking the dual we obtain $$c^T(T \tilde Y)=(1-t) \prod_i (1+a_it+x_i)=Q(t)$$

The same argument works to prove $c^T(Y)=Q(0)$.
\end{proof}

The theorem follows.

\bibliographystyle{amsalpha}
\bibliography{references}

\providecommand{\bysame}{\leavevmode\hbox to3em{\hrulefill}\thinspace}
\providecommand{\MR}{\relax\ifhmode\unskip\space\fi MR }
\providecommand{\MRhref}[2]{%
  \href{http://www.ams.org/mathscinet-getitem?mr=#1}{#2}
}
\providecommand{\href}[2]{#2}
\begin{thebibliography}{DLPV21}

\bibitem[Alu06]{Aluffi-Chern-singular}
Paolo Aluffi, \emph{Classes de {C}hern des vari\'{e}t\'{e}s singuli\`{e}res,
  revisit\'{e}es}, C. R. Math. Acad. Sci. Paris \textbf{342} (2006), no.~6,
  405--410. \MR{2209219}

\bibitem[Alu10]{Aluffi-Chern}
\bysame, \emph{Chern classes of blow-ups}, Math. Proc. Cambridge Philos. Soc.
  \textbf{148} (2010), no.~2, 227--242. \MR{2600139}

\bibitem[Alu11]{Aluffi-lecture}
\bysame, \emph{Chern classes of blowups}, 2011.

\bibitem[DLPV21]{dilorenzo-pernice-vistoli}
Andrea Di~Lorenzo, Michele Pernice, and Angelo Vistoli, \emph{Stable cuspidal
  curves and the integral chow ring of {$\overline{\mathcal{M}}_{2,1}$}}, 2021.

\bibitem[EG98]{EG}
Dan Edidin and William Graham, \emph{Equivariant intersection theory}, Invent.
  Math. \textbf{131} (1998), no.~3, 595--634. \MR{1614555}

\bibitem[EH16]{eisenbud_harris}
David Eisenbud and Joe Harris, \emph{3264 and all that: A second course in
  algebraic geometry}, Cambridge University Press, 2016.

\bibitem[Ful93]{Fulton-toric}
William Fulton, \emph{Introduction to toric varieties}, Annals of Mathematics
  Studies, vol. 131, Princeton University Press, Princeton, NJ, 1993, The
  William H. Roever Lectures in Geometry. \MR{1234037}

\bibitem[Ful98]{fulton}
\bysame, \emph{Intersection theory}, second ed., Ergebnisse der Mathematik und
  ihrer Grenzgebiete. 3. Folge. A Series of Modern Surveys in Mathematics
  [Results in Mathematics and Related Areas. 3rd Series. A Series of Modern
  Surveys in Mathematics], vol.~2, Springer-Verlag, Berlin, 1998. \MR{1644323}

\bibitem[Inc22]{inchiostro}
Giovanni Inchiostro, \emph{Moduli of genus one curves with two marked points as
  a weighted blow-up}, Mathematische Zeitschrift \textbf{302} (2022), no.~3,
  1905–1925.

\bibitem[Iwa07]{Iwanari}
Isamu Iwanari, \emph{Integral chow rings of toric stacks}, International
  Mathematics Research Notices \textbf{2009} (2007), 4709--4725.

\bibitem[Kee92]{Keel}
S.~Keel, \emph{Intersection theory of moduli space of stable n-pointed curves
  of genus zero}, Transactions of the American Mathematical Society
  \textbf{330} (1992), 545--574.

\bibitem[MM12]{mustata-mustata}
Anca~M. Mustaţă and Andrei Mustaţǎ, \emph{The structure of a local
  embedding and chern classes of weighted blow-ups}, Journal of the European
  Mathematical Society \textbf{14} (2012), 1739--1794.

\bibitem[MRV06]{MV}
Luis~Alberto Molina~Rojas and Angelo Vistoli, \emph{On the {C}how rings of
  classifying spaces for classical groups}, Rend. Sem. Mat. Univ. Padova
  \textbf{116} (2006), 271--298. \MR{2287351}

\bibitem[QR21]{quek-rydh-weighted-blow-up}
Ming~Hao Quek and David Rydh, \emph{Weighted blow-ups}, In preparation, draft
  available at \url{https://people.kth.se/~dary/weighted-blowups20220329.pdf},
  2021.

\bibitem[{Sta}]{stacks-project}
The {Stacks Project Authors}, \emph{{S}tacks {P}roject},
  \url{http://stacks.math.columbia.edu}.

\bibitem[Tot14]{totaro}
Burt Totaro, \emph{Group cohomology and algebraic cycles}, Cambridge Tracts in
  Mathematics, Cambridge University Press, 2014.

\bibitem[Vis89]{vistoli}
Angelo Vistoli, \emph{Intersection theory on algebraic stacks and on their
  moduli spaces}, Inventiones mathematicae \textbf{97} (1989), 613--670.

\end{thebibliography}

\end{document}